\documentclass[12pt]{amsart}

\usepackage{latexsym,amssymb,amsfonts,amsmath,mathrsfs,hyperref,mathtools}

\calclayout

\usepackage{pifont}%
\usepackage{enumerate}
\hypersetup{colorlinks,citecolor=NavyBlue,filecolor=black,linkcolor=black,urlcolor=black}
\usepackage[dvipsnames]{xcolor}
\usepackage{comment} 
\usepackage[capitalize]{cleveref}

\newcommand{\vertiii}[1]{{\left\vert\kern-0.25ex\left\vert\kern-0.25ex\left\vert #1 
		\right\vert\kern-0.25ex\right\vert\kern-0.25ex\right\vert}}

  \DeclareMathOperator{\cb}{cb}
  \renewcommand{\cb}{{\mathrm{cb}}}
\DeclareMathOperator{\fa}{{\mathrm{fact}}}
     
  \newcommand{\norm}[1]{\|#1\|}

  \newcommand{\Id}{\ensuremath{\mathop{\rm Id}\nolimits}}

  \newcommand{\poly}{\mbox{\rm poly}}

  \DeclareMathOperator{\Diag}{Diag}

  \newcommand{\ind}[1]{\mathbf{#1}}
   
  \newcommand{\beq}{\begin{equation}}
  \newcommand{\eeq}{\end{equation}}
  \newcommand{\beqn}{\begin{equation*}}
  \newcommand{\eeqn}{\end{equation*}}
  \newcommand{\beqr}{\begin{eqnarray}}
  \newcommand{\eeqr}{\end{eqnarray}}
  \newcommand{\beqrn}{\begin{eqnarray*}}
  \newcommand{\eeqrn}{\end{eqnarray*}}
  \newcommand{\bmline}{\begin{multline}}
  \newcommand{\emline}{\end{multline}}
  \newcommand{\bmlinen}{\begin{multline*}}
  \newcommand{\emlinen}{\end{multline*}}

  \theoremstyle{plain}
  \newtheorem{theorem}{Theorem}[section]

  \theoremstyle{definition}

  \theoremstyle{remark}
  \newtheorem{remark}[theorem]{Remark}
  
  \renewenvironment{proof}[1][]{
    	\begin{trivlist}
     	\item[\hspace{\labelsep}{\em\noindent Proof#1:\/}]}
     	{{\hfill$\Box$}
    	\end{trivlist}
  }
  
\usepackage{thm-restate} 

\begin{document}

\title[Understanding CS factorization through SDP]{Understanding Christensen-Sinclair factorization via semidefinite programming
}

\author[F. Escudero-Guti\'errez]{Francisco Escudero-Guti\'errez}
\address{CWI \& QuSoft, Science Park 123, 1098 XG Amsterdam, The Netherlands}
\email{feg@cwi.nl}

\thanks{
This research was supported by the European Union’s Horizon 2020 research and innovation programme under the Marie Sk{\l}odowska-Curie grant agreement no. 945045, and by the NWO Gravitation project NETWORKS under grant no. 024.002.003.}

\begin{abstract}
    We show that the Christensen-Sinclair factorization theorem, when the underlying Hilbert spaces are finite dimensional, is an instance of strong duality of semidefinite programming. This gives an elementary proof of the result and also provides an efficient  algorithm to compute the Christensen-Sinclair factorization. 
\end{abstract}
    
\maketitle

\section{Introduction}
The seminal result in operator space theory of Christensen and Sinclair establishes that a 
$t$-linear form $T:(M_n)^t\to \mathbb C$  is completely contractive if and only if it can be factored into a sequence of $*$-representations interlaced with contractions~\cite{christensen1987representations}. 
Here, $M_n$ is the space of complex $n\times n$ matrices equipped with the operator norm of matrices when they are regarded as linear maps from $\ell_2^n(\mathbb C)$ 
to $\ell_2^n(\mathbb C).$
The completely bounded norm of $T$ is defined in the following way.  For every $m\in \mathbb N$, consider the map $T_m: (M_{nm})^t\to M_m$ defined via $$T_m(X_1,\dots,X_t)=\left(\sum_{r_1,\dots,r_{t-1}\in [m]}T((X_1)_{i,r_1},(X_2)_{r_1,r_2},\dots (X_t)_{r_{t-1},j})\right)_{i,j\in [m]},$$
where $(X_1)_{ij}$ is the $(i,j)$-th block of $X_1$ when $X_1$ is regarded as a matrix with $n\times n$ blocks of size $m\times m$. Then, the completely bounded norm of $T$ is given by $$\norm{T}_{\cb}=\sup\{\norm{T_m}: m\in\mathbb N\},$$
where $\norm{T_m}$ is the operator norm of $T_m$.
\begin{theorem}\label{theo:CS}(Christensen-Sinclair)
    Let $T:(M_n)^t\to \mathbb C$ be a $t$-linear form. Then, $\norm{T}_{\cb}\leq 1$ if and only if there exist $d\in\mathbb N$, unit vectors $u,v\in\mathbb C^d$ and matrices $A_0\in M_{d,dn}$, $A_{1},\dots,A_{t-1}\in M_{dn,dn}$ and $A_{t}\in M_{dn,d}$ with operator norm at most 1 such that 
    \begin{equation}\label{eq:CSfactorization}
        T(X_1,\dots, X_t)=\langle u, A_0(X_1\otimes \Id_d)A_1\dots A_{t-1}(X_t\otimes \Id_d)A_tv\rangle,
    \end{equation}
    for every $X_1,\dots,X_t\in M_n.$
\end{theorem}
To the best of our knowledge, all the proofs of \cref{theo:CS} are done in more generality, considering the space of endomorphisms of some complex Hilbert space instead of $M_n$. Even when specified to the finite-dimensional case, they require tools from operator spaces, several applications of the Hahn-Banach theorem, going through infinite-dimensional spaces and are not constructive \cite{christensen1987representations,PAULSEN1987258,BLECHER1991262}. 

By contrast, we give an elementary and constructive proof that does not require the use of Hahn-Banach. We do this by showing that \cref{theo:CS} is an instance of the strong duality of semidefinite programming. In the proof we use as a black-box a standard and elementary result due to Slater that ensures strong duality. We stress that the proof of Slater's theorem just requires a finite-dimensional separation theorem, but no application of the Hahn-Banach theorem (for a proof see \cite{watrous2009semidefinite}). Assuming Slater's theorem, our argument just requires simple notions of linear algebra. 

Semidefinite programming is an extension of linear programming that can model a bigger family of problems and can still be efficiently solved up to arbitrary precision (see \cite{LAURENT2005393} for an introduction to semidefinite programming). To be more precise, let $H_N$ be the space of Hermitian matrices of $M_N$ and let $H_N^+$ be the cone of positive semidefinite matrices. A collection of matrices $C,B_1,\dots, B_L\in H^N$ and a vector $b\in \mathbb{R}^L$  define a \emph{primal semidefinite program (P)} and a \emph{dual semidefinite program (D)}, which in their \emph{canonical form} are given by 
\begin{align}\label{eq:SDPDef}
	(P)\ \ &\inf &&\ \langle C,Y\rangle &&&\ (D)&&&&  \sup &&&&&\langle b,y\rangle\\
	&\text{s.t.}\ &&   Y\in H_N^+ &&& &&&& \text{s.t.} &&&&& y\in\mathbb{R}^L\nonumber\\
	 &  &&\mathcal{B}(Y)=b &&&  &&&& &&&&& C-\mathcal{B}^*(y)\in H_N^+,\nonumber
\end{align} 
where $\mathcal{B}:H_N\to \mathbb{R}^L$ is given by $\mathcal{B}(Y):=(\langle B_1,Y\rangle,\dots,\langle B_L,Y\rangle)$, $\mathcal{B}^*(y)=\sum_{i\in [L]}y_iB_i$ and $\langle B,Y\rangle=\text{Tr}(BY)$. 
It is always satisfied that the optimal value of $(P)$ is at least the optimal value of $(D)$, what is known as \emph{weak duality}. In addition, under some mild assumptions provided by Slater's theorem (see \cref{theo:slater} below), both values are equal, what is known as \emph{strong duality}.\footnote{Usually semidefinite programs are phrased in terms of real symmetric  matrices, but the weak and strong dualities hold when one substitutes symmetric matrices by Hermitian ones, as they are general properties of conic programs (see for instance \cite{watrouslecturenotes}).} Note that if all matrices $C,B_1,\dots, B_L$ were diagonal, $(P)$ and $(D)$ would be linear programs.

It will be convenient  to introduce the factorization norm of a $t$-linear form $T:(M_n)^t\to\mathbb C$, which is given by 
\begin{align*}
    \norm{T}_{\fa}=&\inf &&w\\
    & \mathrm{s.t.} &&  T(X_1,\dots, X_t)=\langle u, A_0(X_1\otimes \Id_d)A_1\dots A_{t-1}(X_t\otimes \Id_d)A_tv\rangle,\\ & && \forall X_1,\dots,X_t\in M_n,\\ 
    & && d\in\mathbb N,\ u,v\in\mathbb C^d,\ \norm{u}_2^2=\norm{v}^2_2=w,\\ 
    & && A_0\in M_{d,dn},\ A_1,\dots,A_{t-1}\in M_{dn,dn},\ A_d\in M_{dn,d}\ \text{contractions}.
\end{align*}

Now, we are ready to state our main result. 

\begin{theorem}\label{theo:main}
    Given a $t$-linear form $T:(M_n)^t\to \mathbb C$, there is a pair of semidefinite programs $(P_{CS})$ and $(D_{CS})$ such that 
    \begin{enumerate}[i) ]
        \item $(P_{CS})$ optimal value equals $\norm{T}_{\fa}$,\label{item:main1}
        \item $(D_{CS})$ optimal value equals $\norm{T}_{\cb}$, \label{item:main2}
        \item $(D_{CS})$ is the dual of $(P_{CS})$ and their optimal values are equal. \label{item:main3}
    \end{enumerate}
\end{theorem}

\cref{theo:main} has three important consequences. The first one is already clear from the statement, and the other two will become clear later (see \cref{rem:Items2and3}). These consequences are:
\begin{enumerate}
    \item \cref{theo:main} implies \cref{theo:CS};\label{item:consquence1}
    \item $(P_{CS})$ and $(D_{CS})$ have $O(\poly(n)^t)$ variables, so the known algorithms to approximate semidefinite programs can be used to efficiently compute the completely bounded norm;\label{item:consequence2}
    \item from the solution returned by these algorithms one can extract a description of the vectors and matrices appearing in a factorization of the kind in \cref{eq:CSfactorization}.\label{item:consequence3}
\end{enumerate}

\subsection{Some remarks}
We note that \cref{theo:main} works when we substitute $M_n$ by any subspace of $M_n$, i.e, when we consider an operator space that inherits its structure from $M_n$, and the algorithmic consequences hold as well. By this we mean that if $V$ is a normed subspace of $M_n$, then, for every $m\in\mathbb N$, $V\otimes M_m$ inherits a norm as a subspace of $M_{mn}$, and that norm defines a notion of completely bounded norm for $t$-linear maps $T:V^t\to \mathbb C$. We also remark that our proof of \cref{theo:main} can be extended to the case where $M_n$ is substituted by the space of endomorphisms of a separable complex space. However, in that case one requires Hahn-Banach to prove strong duality and the algorithmic consequences would be lost. It is not clear to us how to extend the proof to the non-separable case, as our technique heavily relies on the existence of a countable orthonormal basis. 

\subsection{Related work.} An analogous result for the linear case ($t = 1$) was proven by Watrous \cite{watrous2009semidefinite,watrous2012simpler}. Gribling and Laurent showed that the completely bounded norm of a form $T:(\ell_{\infty}^n(\mathbb C))^t\to \mathbb C$ can be computed as a semidefinite program, so item  \eqref{item:consequence2} above, for the special case of $\ell_{\infty}^n(\mathbb C)$, also follows from their work \cite{gribling2019semidefinite}. The initial motivation of our  work comes from the applications of the Christensen-Sinclair theorem to quantum information. Arunachalam, Bri\"et and Palazuelos used this result to give a characterization of quantum query complexity, showing that every quantum query algorithm is a completely contractive form, and vice versa \cite{QQA=CBF}. However, prior to this work, given a completely contractive form there was no way of obtaining the corresponding algorithm, which  in our context is equivalent to finding the matrices that define a factorization as in \cref{eq:CSfactorization}. Item \eqref{item:consequence3} fills this gap. Although conceptually different, the proof of the operator space Grothendieck theorem by Regev and Vidick is similar in spirit to our result. First, they applied that theorem of operator spaces to prove results in the context of quantum information and complexity theory \cite{naor2013efficient,regev2015quantum}. Then, they reproved it in an elementary way using tools from quantum information~\cite{regev2014elementary}.

\section{Proof} 

\subsection{Notation} We use $[n]$ to denote $\{1,\dots,n\}$. Sometimes we will consider multi-indices belonging to $([n]\times [n])^s$ and refer to them as $\ind{I}$, where $I_1,\dots,I_s\in [n]\times [n]$. $M_{m,n}$ is the space of $m\times n$ complex matrices and $M_n=M_{n,n}$. $\{E_{i,j}\}_{i,j\in [n]}$ is the canonical basis of $M_n$, where the $(i',j')$-th entry of $E_{i,j}$ is given by $\delta_{i,i'}\delta_{j,j'}$.  Given a matrix $X\in M_{n}$ and $I=(i,j)\in [n]\times [n]$, we say that $X_I$ is the coordinate of $X$ in the former basis, corresponding to $E_{i,j}$. Sometimes we will consider $X\in M_{mn}$ and $I\in [n]\times [n]$, in which case $X(I)$ will be $I$-th block when we regard $X$ as $n\times n$ block-matrix with blocks of size $m\times m$. Given $s$ matrices $X_1,\dots,X_s\in M_n$, we denote by $\ind{X}$ the matrix-vector defined by $(X_1,\dots,X_s)$. Given $\ind{I}\in ([n]\times [n])^s$ and $\ind{X}\in (M_{mn})^s$, $\ind{X}(\ind{I})$ is defined as the matrix product given by $X_1(I_1)\dots X_t(I_t)$. The norm of vectors of $\mathbb{C}^d$ is the Euclidean norm, and the norm of matrices is the operator norm when considered them as linear maps from $\ell_2(\mathbb C)$ to $\ell_2(\mathbb C)$. We say that a matrix is a contraction if its norm is at most 1. Given a vector $\alpha\in\mathbb{C}^n$, $\Diag(\alpha)$ is the diagonal matrix of $M_n$ whose diagonal is $\alpha$. $\Id_d$ is the identity matrix of $M_d$. Given $z\in \mathbb C$, $\Re z\in \mathbb R$ is its real part and $\Im z\in \mathbb R$ its imaginary part. 

We will often identify a $t$-form $T:(M_n)^t\to \mathbb C$ with its tensor of coefficients $(T_{\ind I})_{\ind I}\in (\mathbb C^{n\times n})^t$ defined via $T_{\ind I}=T(E_{I_1},\dots,E_{I_t}),\ \text{so}\ T(\ind X)=\sum_{\ind {I}\in ([n]\times [n])^t}T_{\ind I}X(\ind I).$ One can also write $T_m:(M_{nm})^t\to\mathbb C$ in terms of these coefficients, namely, $T_m(\ind X)=\sum_{\ind i\in ([n]\times [n])^t}T_{\ind I} X(\ind I)$; where, given $I\in [n]\times [n]$ and $X\in M_{nm}$, $\ind X(I)$ is the $I$-th $m\times m$-dimensional block of $X$ when regarded as a block-matrix with $n\times n$ blocks.

We recall the reader that for $Y\in M_N$, $Y\succeq 0$ if and only if there exist vectors $v_i$ such that $Y_{(i,j)}=\langle v_i,v_j\rangle$ for every $i,j\in [n]$. In that case, we say that $Y$ is the Gram matrix of $\{v_i\}_{i\in [n]}.$

\subsection{The primal semidefinite program}
In this section, we introduce $(P_{CS})$ and prove \cref{theo:main} \eqref{item:main1}.  Before doing that we give some intuition of why $\norm{T}_{\fa}$ can be formulated as a semidefinite program. Assume that $T$ factors as in \cref{eq:CSfactorization} with vectors of $u$ and $v$ of square norm $w.$ Then, we consider the following block structure for the contractions $A_s$: 
	\begin{align}
		A_0=\begin{pmatrix}
			A_0(1)\\
			\vdots\\
			A_0(n)
		\end{pmatrix},\ A_s=\begin{pmatrix}
		A_s(1,1) & \dots & A_s(1,n)\\ \vdots & \ddots & \vdots\\ A_s(n,1)& \dots & A_s(n,n)
	\end{pmatrix},\ A_t=\begin{pmatrix}
	A_t(1)\\ \dots\\ A_t(n)
\end{pmatrix},\label{eq:blockstructure}
\end{align}
We define the following vectors, 
\begin{align}
    v_i&=A_t(i)v,\ \mathrm{for}\ i\in [n],\label{eq:defvi}\\ 
    v_{\ind{I}}&=A_{t-s}(I_1)\dots A_{t-1}(I_{s})A_t(I_{s+1})v,\ \mathrm{for}\ \ind{I}\in ([n]\times [n])^{s}\times [n],\ s\in [t-1],\label{eq:defvI1}\\
    v_{\ind{I}}&=A_0(I_1)A_1(I_2)\dots A_t(I_{t+1}), \mathrm{for}\ \ind{I}\in [n]\times ([n]\times [n])^{t-1}\times [n].\label{eq:defVI2}
\end{align}
Now, one should note that $T_{\ind I}=\langle u,v_{\ind I}\rangle $ for every $\ind I\in ([n]\times [n])^t$, so $T_{\ind I}$ is encoded in the entries of  $Y=$Gram$\{u,v_{\ind I}\}$ (which corresponds to \eqref{eq:approx1} below). In addition, the fact that $A_i$ are contractions can be encoded in the entries of this Gram matrix (which gives rise to \cref{eq:A_0contr,eq:A_scontr,eq:A_tcontr} below). With these intuitions, we are ready to state $(P_{CS})$: 
\begin{align}\label{eq:CSinSDP}
	& \inf &&w\tag{$P_{CS}$}\\
	&\mathrm{s.t.} && w\geq 0,\ Y\succeq 0,\label{eq:PSDofConstraintofCSinSDP}\\
	& &&\Re Y_{0,\ind{I}}= \Re T_{\ind I},\ \Im Y_{0,\ind{I}}=\Im T_{\ind I},\ \ind I\in ([n]\times [n])^t,\label{eq:approx1}\\
	& &&Y_{0,0}=w,\label{eq:uunit}\\
	& && \sum_{i\in [n]} Y_{i,i}\leq w,\label{eq:A_tcontr}\\
	& &&  \sum_{i\in [n]}(Y_{(i,j)\ind{J}, (i,j')\ind{J}'})_{j,j'\in[n], \ind{J},\ind{J}'\in ([n]\times [n])^{s-1}\times [n]}\label{eq:A_scontr}\preceq\oplus_{k\in [n]}(Y_{\ind{J}, \ind{J}'})_{\ind{J},\ind{J}'\in ([n]\times [n])^{s-1}\times [n]},\  s\in [t-1],\\
	& && \sum_{i\in [n]}(Y_{i\ind{J}, i\ind{J}'})_{\ind{J},\ind{J}'\in ([n]\times [n])^{t-1}\times [n]}\preceq (Y_{\ind{J}, \ind{J}'})_{\ind{J},\ind{J}'\in ([n]\times [n])^{t-1}\times [n]},\label{eq:A_0contr}
\end{align}
where $Y\in M_{D}$ and $D=1+n+n^3+\dots+n^{2t-3}+n^{2t-1}+n^{2t}$. The rows and columns of $Y$ are labeled by the elements of $\{0\}\cup [n]\cup [n]^3\cup\dots [n]^{2t-3}\cup [n]^{2t-1}\cup [n]^{2t}$.

\begin{proof}[ of \cref{theo:main} \eqref{item:main1}] Note that \cref{eq:approx1} ensures that $(Y_{0,\ind{I}})_{\ind{I}}$ equals $(T_{\ind I})_{\ind I}$. Thus, it suffices to show that a tensor $(R_{\ind I})_{\ind I}$ satisfies $\norm{R}_{\fa}=w$ if and only if there is a matrix $Y$  that satisfies \cref{eq:uunit,eq:A_tcontr,eq:A_scontr,eq:A_0contr}
 and $Y_{0,\ind I}=R_{\ind I}$. 

Assume first that $(R_{\ind I})_{\ind I}$ factors as in  \cref{eq:CSfactorization} for some vectors with $\norm{u}^2=\norm{v}^2=w$. Then, consider the block structure for the contractions $A_s$ given in \cref{eq:blockstructure}, and define the vectors $\{u,v_{\ind{I}}:\ \ind{I}\in [n]\cup_{s\in [t-1]} ([n]\times [n])^s\times [n]\cup [n]\times ([n]\times [n])^{t-1}\times [n]\}$ as in \cref{eq:defvi,eq:defvI1,eq:defVI2}. 
	Then, $R_{\ind{I}}=\langle u,v_{\ind{I}}\rangle,$
	for every $\ind{I}\in ([n]\times [n])^t$. This way, if we consider the positive semidefinite matrix $$Y:=\mathrm{Gram}\{u,v_{\ind{I}}:\ \ind{I}\in [n]\cup_{s\in [t-1]} ([n]\times [n])^s\times [n]\cup [n]\times ([n]\times [n])^{t-1}\times [n]\},$$ and we label the rows and columns corresponding to $u$ with $0$ and the ones corresponding to $v_\ind{I}$ with $\ind{I}$, we have that $R_{\ind{I}}=Y_{0,\ind{I}}$
	for every $\ind{I}\in ([n]\times [n])^{t}.$  \cref{eq:uunit} follows from the fact that $\norm{u}^2=w.$ From the fact that $A_t$ is a contraction, \cref{eq:A_tcontr} follows: 
	\begin{equation*}\label{eq:A_tcontr2}
		\sum_{i\in [n]}Y_{i,i}=\sum_{i\in [n]}\langle v_i,v_i\rangle=\langle v,\sum_{i\in [n]}A(i)^\dagger A(i)v=\langle v,A_t^\dagger A_tv\rangle \leq \langle v, v\rangle=w.
	\end{equation*} 
	From the fact that $A_s$ are contractions for $s\in [t-1]$ \cref{eq:A_scontr} follows. 
	Indeed, let $\lambda\in \mathbb{C}^{[n]\times([n]\times [n])^{s-1}\times [n]}$. Then, 
	\begin{align*}
		&\langle \lambda,\sum_{i\in [n]}(Y_{(i,j)\ind{J}, (i,j')\ind{J}'})_{j,j'\in[n], \ind{J},\ind{J}'\in ([n]\times [n])^{s-1}\times [n]}\lambda\rangle \\
		&=\sum_{i\in [n],j,j'\in [n],\ind{J},\ind{J}'\in  ([n]\times [n])^{s-1}\times [n]}\lambda_{j\ind{J}}^*\langle v_{(i,j)\ind{J}},v_{(i,j')\ind{J}'}\rangle\lambda_{j'\ind{J}'}\\
		&=\sum_{j,j'\in [n],\ind{J},\ind{J}'\in  ([n]\times [n])^{s-1}\times [n]}\lambda_{j\ind{J}}^*\langle v_{\ind{J}},\left(\sum_{i\in [n]}A_{t-s}^\dagger (j,i)A_{t-s}(i,j')\right)v_{\ind{J}'}\rangle\lambda_{j'\ind{J}'}\\
		&=\sum_{j,j'\in [n],\ind{J},\ind{J}'\in  ([n]\times [n])^{s-1}\times [n]}\lambda^*_{j\ind{J}}\langle v_{\ind{J}},(A_{t-s}^\dagger A_{t-s})(j,j')v_{\ind{J}'}\rangle \lambda_{j'\ind{J}'}\\
		&\leq \sum_{j,j'\in [n],\ind{J},\ind{J}'\in  ([n]\times [n])^{s-1}\times [n]}\lambda_{j\ind{J}}^*\langle v_{\ind{J}},\delta_{j,j'}v_{\ind{J}'}\rangle\lambda_{j'\ind{J}'}\\
		&=\sum_{j\in [n],\ind{J},\ind{J}'\in  ([n]\times [n])^{s-1}\times [n]}\lambda^*_{j\ind{J}}(Y_{\ind{J}, \ind{J}'})_{\ind{J},\ind{J}'\in ([n]\times [n])^{s-1}\times [n]}\lambda_{j\ind{J}'}\\
		&=\langle \lambda,\oplus_{k\in [n]}(Y_{\ind{J}, \ind{J}'})_{\ind{J},\ind{J}'\in ([n]\times [n])^{s-1}\times [n]}\lambda\rangle ,
	\end{align*} 
	where in the second equality we have used that $A_{t-s}(i,j)^\dagger=A_{t-s}^\dagger (j,i)$ and in the inequality we have used that $A_{t-s}$ is a contraction, namely that $A_{t-s}^\dagger A_{t-s}\preceq \Id_{dn}$. The fact that $A_0$ is a contraction implies \cref{eq:A_0contr} and its
	 check follows the same philosophy as the one of \cref{eq:A_scontr}. 
	
	Now assume that there exists $Y\succeq 0$,  satisfying equations \cref{eq:uunit,eq:A_tcontr,eq:A_scontr,eq:A_0contr} and $R_{\ind{I}}=Y_{0,\ind{I}}$. Consider $d\in \mathbb{N}$ and vectors $\{u,v_{\ind{I}}:\ \ind{I}\in [n]\cup_{s\in [t-1]} ([n]\times [n])^s\times [n]\cup [n]\times ([n]\times [n])^{t-1}\times [n]\}\in \mathbb{C}^d$ such that $$Y=\mathrm{Gram}\{u,v_{\ind{I}}:\ \ind{I}\in [n]\cup_{s\in [t-1]} ([n]\times [n])^s\times [n]\cup [n]\times ([n]\times [n])^{t-1}\times [n]\},$$
	where again we label by $0$ the rows and columns corresponding to $u$ and by $\ind{I}$ the ones corresponding to $v_{\ind{I}}$. \cref{eq:uunit} implies that $\norm{u}^2=w$. We define $A_t$ through its blocks. Let $v\in\mathbb{C}^d$ be a vector with $\norm{v}^2=w$.  $A_t(i)\in M_d$ is defined as the matrix that maps $v$ to $v_i$ and is extended by $0$ to the orthogonal complement of $\text{span}\{v\}$. This way, $A_t$ is a contraction, because 
	\begin{align*}
		\norm{A_t}^2=\frac{\langle v,A_t^\dagger A_tv\rangle}{w}=\frac{1}{w}\sum_{i\in [n]}\langle v,A_t(i)^\dagger A_t(i)v\rangle =\frac{1}{w}\sum_{i\in [n]}\langle v_i,v_i\rangle =\frac{1}{w}\sum_{i\in [n]}Y_{i,i}\leq 1,
	\end{align*} 
	where in the inequality we have used \cref{eq:A_tcontr}. The definition of $A_{t-s}$ for $s\in [t-1]$ is slightly more complicated. The block $A_{t-s}(i,j)$ is defined as linear map on $\text{span}\{v_{\ind{I}}:\ind{I}\in ([n]\times [n])^{s-1}\times [n]\}$ by $$A_{t-s}(i,j)v_{\ind{J}}=v_{(i,j)\ind{J}}$$ and extended by $0$ to the orthogonal complement. First, we have to check that this a good definition, namely that for every $\lambda\in\mathbb{C}^{n^{2(s-1)+1}}$ $$\sum_{\ind{J}\in ([n]\times [n])^{(s-1)}\times [n]}\lambda_\ind{J}v_{(i,j)\ind{J}}=0\implies \sum_{\ind{J}\in ([n]\times [n])^{(s-1)}\times [n]}\lambda_\ind{J}v_{\ind{J}}=0.$$
	Indeed, we can prove something stronger. For any $\lambda\in\mathbb{C}^{n^{2(s-1)+1}}$, we define $\tilde{\lambda}\in \mathbb{C}^{n^{1+2(s-1)+1}}$ by $\tilde{\lambda}_{j'\ind{J}}:=\delta_{j,j'}\lambda_{\ind{J}}$. Hence, 
	\begin{align*}
		&\langle \sum_{\ind{J}\in ([n]\times [n])^{s-1}\times [n]}\lambda_\ind{J}v_{(i,j)\ind{J}},\sum_{\ind{J}'\in ([n]\times [n])^{s-1}\times [n]}\lambda_{\ind{J}'}v_{(i,j)\ind{J}'}\rangle\\
		&=\langle \lambda, (Y_{(i,j)\ind{J}, (i,j)\ind{J}'})_{\ind{J},\ind{J}'\in ([n]\times [n])^{s-1}\times [n]}\lambda\rangle \\
		&=\langle \tilde{\lambda},(Y_{(i,j')\ind{J}, (i,j'')\ind{J}'})_{j',j''\in[n], \ind{J},\ind{J}'\in ([n]\times [n])^{s-1}\times [n]},\tilde{\lambda}\rangle\\
		&\leq \langle \tilde{\lambda},\sum_{i'\in [n]}(Y_{(i',j')\ind{J}, (i',j'')\ind{J}'})_{j',j''\in[n], \ind{J},\ind{J}'\in ([n]\times [n])^{s-1}\times [n]}\tilde{\lambda}\rangle\\
		&\leq\langle \tilde{\lambda},\oplus_{k\in [n]}(Y_{\ind{J}, \ind{J}'})_{\ind{J},\ind{J}'\in ([n]\times [n])^{s-1}\times [n]}\tilde{\lambda}\rangle\\
		&=\langle \lambda,(Y_{\ind{J}, \ind{J}'})_{\ind{J},\ind{J}'\in ([n]\times [n])^{s-1}\times [n]}\lambda\rangle\\
		&=\langle \sum_{\ind{J}\in ([n]\times [n])^{s-1}\times [n]}\lambda_\ind{J}v_{\ind{J}},\sum_{\ind{J}'\in ([n]\times [n])^{s-1}\times [n]}\lambda_{\ind{J}'}v_{\ind{J}'}\rangle,
	\end{align*}
	where in the first inequality we have used that $$(Y_{(i',j')\ind{J}, (i',j'')\ind{J}'})_{j',j''\in[n], \ind{J},\ind{J}'\in ([n]\times [n])^{s-1}\times [n]}\succeq 0$$ for every $i'\in [n]$, and in the second inequality we have used \eqref{eq:A_scontr}. Now, we have to check that $A_{t-s}$ is a contraction. By the definition of $A_{t-s}$, we just have to check that for every $\lambda\in\mathbb{C}^{n^{1+2(s-1)+1}},$ $$\lambda v:=\begin{pmatrix}
		\sum_{\ind{J}_1\in ([n]\times [n])^{s-1}\times [n]}\lambda_{1\ind{J}_1}v_{\ind{J}_1}\\
		\vdots\\
		\sum_{\ind{J}_n\in ([n]\times [n])^{s-1}\times [n]}\lambda_{n\ind{J}_n}v_{\ind{J}_n}
	\end{pmatrix}$$
	is mapped to a vector with smaller or equal norm. Indeed, 
	\begin{align*}
		\langle A_{t-s}\lambda v,A_{t-s}\lambda v\rangle &=\sum_{i,j,j'\in [n],\ind{J},\ind{J'}\in [n]^{2(s-1)+1}}\lambda_{j\ind{J}}\langle v_{(i,j)\ind{J}}, v_{(i,j')\ind{J}'}\rangle \lambda_{j'\ind{J}'}\\
		&=\langle \lambda,\sum_{i\in [n]}(Y_{(i,j)\ind{J}, (i,j')\ind{J}'})_{j,j'\in[n], \ind{J},\ind{J}'\in ([n]\times [n])^{s-1}\times [n]}\lambda\rangle \\
		&\leq \langle \lambda,\oplus_{k\in [n]}(Y_{\ind{J}, \ind{J}'})_{\ind{J},\ind{J}'\in ([n]\times [n])^{s-1}\times [n]}\lambda\rangle\\
		&=\langle \lambda v,\lambda v\rangle,
	\end{align*}
	where in the inequality we have used \cref{eq:A_scontr}. Finally, we define $A_0$ through its blocks. $A_0(i)\in M_d$ is defined by $A_0(i)v_{\ind{J}}=v_{j\ind{J}}$ and extended by $0$ to the orthogonal complement of $\text{span}\{v_{\ind J}: \ind J\in [n]\times ([n]\times [n])^{t-1}\}$. Using \cref{eq:A_0contr}, we can check that these blocks are well-defined and that $A_0$ is a contraction using a similar argument to the one that we have just used to verify the same properties of $A_{t-s}$. 
\end{proof}
\begin{remark}\label{rem:Items2and3}
    $(P_{CS})$ has $O(\poly(n)^t)$ variables, so item \eqref{item:consequence2} holds. Item \eqref{item:consequence3} can be inferred from the proof of \cref{theo:main} \eqref{item:main1}, where a recipe to extract a factorization as in \cref{eq:CSfactorization} for $(Y_{0,\ind I})_{\ind I}$ satisfying \cref{eq:uunit,eq:A_0contr,eq:A_scontr,eq:A_tcontr} is given. 
\end{remark}

\subsection{The dual semidefinite program} In this section, we introduce $(D_{CS})$ and prove \cref{theo:main} \eqref{item:main2}.  $(D_{CS})$ is given by: 
\begin{align}\label{eq:TcbinSDP}
	&\sup &&\tag{$D_{CS}$} 2\sum_{\ind{I}\in([n]\times [n])^{t}} \Re (T_{\ind{I}})\Re (R_{\ind{I}})+\Im (T_{\ind{I}})\Im (R_{\ind{I}})\\
	&\mathrm{s.t.} && y_0, y_0'\geq 0,\ \begin{pmatrix}
		y_{\ind{J},\ind{J}'}
	\end{pmatrix}_{\ind{J},\ind{J}'\in ([n]\times [n])^{s}}\succeq 0,\ \mathrm{for}\ s\in[t],\label{eq:psdconditions}\\
& && y_0+y_0'\leq 1,\label{eq:ancillary}\\
	& && \mathrm{Diag}(y_0,\dots,y_0)\succeq \sum_{k\in [n]}
		(y_{ki,kj})_{i,j\in [n]},\label{eq:bigcontraction0}\\
 & && 
 	\oplus_{k\in [n]}\begin{pmatrix}
 		y_{\ind{J},\ind{J}'}
 	\end{pmatrix}_{\ind{J},\ind{J}'\in ([n]\times [n])^{s}}\succeq \sum_{i\in [n]}\begin{pmatrix}y_{(ij)\ind{J},(ij')\ind{J'}}\end{pmatrix}_{j,j'\in [n],\ind{J},\ind{J}'\in ([n]\times [n])^{s}},\label{eq:bigcontractions}\\
 & &&\mathrm{for\ }s\in [t-1],\nonumber\\
 & &&\left(\begin{array}{c|ccc}
 	y_0' & \dots &(R_{\ind{J}})_{\ind{J}\in([n]\times [n])^t}/2 & \dots\\ \hline
 	\vdots & & & \\ 
 	\frac{(R_{\ind{J}}^*)_{\ind{J}\in([n]\times [n])^t}}{2} & & \begin{pmatrix}
 		y_{\ind{J},\ind{J}'}
 	\end{pmatrix}_{\ind{J},\ind{J}'\in ([n]\times [n])^{t}}& \\
 \vdots & & &
 \end{array}
 \right)\succeq 0,\label{eq:RinGram2}
\end{align}

\begin{proof}[ of \cref{theo:main} \eqref{item:main2}]
	First, we note that \cref{eq:psdconditions} is the same as saying that there exist $d\in \mathbb{N}$ and vectors $\{u,v,v_{\ind{I}}:\ind{I}\in ([n]\times [n])^s,\ s\in [t]\}\in \mathbb{C}^d$ such that $y_0'=\langle u,u\rangle $, $y_0=\langle v,v\rangle $ and $y_{\ind{I},\ind{I}'}=\langle v_\ind{I},v_{\ind{I}'}\rangle.$ Then, \cref{eq:RinGram2} means that $R_{\ind{I}}=2\langle u,v_{\ind{I}}\rangle $ for every $\ind{I}\in ([n]\times [n])^t$. Next, we will show that \cref{eq:bigcontraction0,eq:bigcontractions} are equivalent to the existence of contractions $A_1,\dots,A_t\in M_{dn}$ such that
	\begin{equation}\label{eq:bigcontr2}
		v_{\ind{I}}=A_{t-s+1}(I_1)\dots A_t(I_s)v,
	\end{equation}
	for every $\ind{I}\in ([n]\times [n])^s$ and every $s\in [t]$, where $A_s(I)\in M_d$ is the $I$-th block of $A_s$ when regarded as a block-matrix with $n\times n$ blocks.  Indeed, assume that \cref{eq:bigcontraction0,eq:bigcontractions} hold. Then,  we define the $I$-th block of $A_{t-s}$ by $$A_{t-s}(I)v_{\ind{J}}:=v_{I\ind{J}}$$ for every $\ind{J}\in ([n]\times[n])^{s}$  and extend it by $0$ on the orthogonal complement of $\text{span}\{v_{\ind J}: \ind J\in ([n]\times [n])^s\}$. Before proving that $A_s$ are contractions, we have to check that $A_s(I)$ are well-defined as linear maps. Namely, that for every $\lambda\in\mathbb{C}^{n^{2s}}$ we have $$\sum_{\ind{J}\in ([n]\times[n])^{s}}\lambda_{\ind{J}}v_{I\ind{J}}=0\implies \sum_{\ind{J}\in ([n]\times[n])^{s}}\lambda_{\ind{J}}v_{\ind{J}}=0.$$ In fact, we can prove something stronger. Let $\lambda\in\mathbb{C}^{n^{2s}}$, and define $\tilde{\lambda}\in \mathbb{C}^{n^{1+2s}}$ by $\tilde{\lambda}_{j'\ind{J}}=\delta_{j,j'}\lambda_{\ind{J}}$. Then 
	\begin{align*}
		\langle \sum_{\ind{J}\in ([n]\times[n])^{s}}\lambda_{\ind{J}}v_{I\ind{J}},\sum_{\ind{J}'\in ([n]\times[n])^{s}}\lambda_{\ind{J}'}v_{I\ind{J}'}\rangle&=\langle\lambda,\begin{pmatrix}
			y_{I\ind{J},I\ind{J}'}
		\end{pmatrix}_{\ind{J},\ind{J}'\in ([n]\times [n])^s}\lambda\rangle\\
	&=\langle\tilde{\lambda},\begin{pmatrix}
		y_{(i,j)\ind{J},(i,j')\ind{J}'}
	\end{pmatrix}_{j,j'\in [n],\ind{J},\ind{J}'\in ([n]\times [n])^s}\tilde{\lambda}\rangle\\
	&\leq \langle\tilde{\lambda},\sum_{i'\in [n]}\begin{pmatrix}
		y_{(i',j)\ind{J},(i',j')\ind{J}'}
	\end{pmatrix}_{j,j'\in [n],\ind{J},\ind{J}'\in ([n]\times [n])^s}\tilde{\lambda}\rangle \\
	&\leq  \langle \lambda,\begin{pmatrix}
		y_{\ind{J},\ind{J}'}
	\end{pmatrix}_{\ind{J},\ind{J}'\in ([n]\times [n])^s}\lambda\rangle\\
	&=\langle \sum_{\ind{J}\in ([n]\times[n])^{s}}\lambda_{\ind{J}}v_{\ind{J}}),\sum_{\ind{J}'\in ([n]\times[n])^{s}}\lambda_{\ind{J}'}v_{\ind{J}'}\rangle,
	\end{align*}
	where in the first inequality we have used that $$\begin{pmatrix}
		y_{(i',j)\ind{J},(i',j')\ind{J}'}
	\end{pmatrix}_{j,j'\in [n],\ind{J},\ind{J}'\in ([n]\times [n])^s}\succeq 0$$ for every $i'\in [n]$, and in the second inequality we have used \cref{eq:bigcontractions} (or \cref{eq:bigcontraction0} if $s=0$). Now we shall prove that $A_s$ are contractions. By their definition, we only have to check that for every $\lambda\in\mathbb{C}^{n^{1+2s}}$ the vector $$\lambda v:=\begin{pmatrix}
	\sum_{\ind{J}\in ([n]\times [n])^s}\lambda_{1\ind{J}}v_{\ind{J}}\\
		\vdots \\
	\sum_{\ind{J}\in ([n]\times [n])^s}\lambda_{n\ind{J}}v_{\ind{J}}
\end{pmatrix}$$
	is mapped through $A_{t-s}$ to a vector of smaller or equal norm. That  is true because 
	\begin{align*}
		\langle A_s\lambda v,A_s\lambda v\rangle &=\sum_{i,j,j'\in [n],\ind{J},\ind{J}'\in ([n]\times [n])^s}\lambda_{j\ind{J}}^*\langle v_{(i,j)\ind{J}},v_{(i,j')\ind{J}'}\rangle\lambda_{j'\ind{J}}\\
		&=\langle \lambda,\sum_{i\in [n]}\begin{pmatrix}y_{(ij)\ind{J},(ij')\ind{J'}}\end{pmatrix}_{j,j'\in [n],\ind{J},\ind{J}'\in ([n]\times [n])^{s}}\lambda\rangle \\
		&\leq \langle\lambda,\begin{pmatrix}
			\begin{pmatrix}
				y_{\ind{J},\ind{J}'}
			\end{pmatrix}_{\ind{J},\ind{J}'\in ([n]\times [n])^{s}} & & \\ & \ddots & \\ & & \begin{pmatrix}
				y_{\ind{J},\ind{J}'}
			\end{pmatrix}_{\ind{J},\ind{J}'\in ([n]\times [n])^{s}}
		\end{pmatrix}\lambda\rangle \\
	&=\langle\lambda v,\lambda v\rangle,
	\end{align*} 
	where in the inequality we have used \cref{eq:bigcontractions} (or \cref{eq:bigcontraction0} in the case of $s=0$).
	
	On the other hand, if \cref{eq:bigcontr2} holds, it is a routine check showing that \cref{eq:bigcontraction0,eq:bigcontractions} hold. Putting everything together, we can rewrite \cref{eq:TcbinSDP} as 
	\begin{align*}
		&\sup && 2\Re(\sum_{\ind{I}\in ([n]\times [n])^t}T_{\ind{I}}R_{\ind{I}}),\\
		&\mathrm{s.t.} &&  R\in (\mathbb C^{n\times n})^{t},\  d\in\mathbb{N},\ u,v\in \mathbb{C}^d, A_s\in M_d\ \mathrm{contractions\ for}\ s\in [n],\\
		& &&\langle u,u\rangle+\langle v,v\rangle \leq 1,\\
		& && R_{\ind{I}}=\langle u,A_1(I_1)\dots A_t(I_t)v\rangle,\ \mathrm{for\ }\ind{I}\in([n]\times [n])^t.
	\end{align*}	
        Note that by multiplying the $v$ of an optimal solution of this optimization program by an appropriate complex phase,  we can express \cref{eq:TcbinSDP} as 

        \begin{align}
		&\sup && 2|\sum_{\ind{I}\in ([n]\times [n])^t}T_{\ind{I}}R_{\ind{I}}|,\label{eq:SDPa}\\
		&\mathrm{s.t.} &&  R\in (\mathbb C^{n\times n})^{t},\  d\in\mathbb{N},\ u,v\in \mathbb{C}^d, A_s\in M_d\ \mathrm{contractions\ for}\ s\in [n],\nonumber\\
		& &&\langle u,u\rangle+\langle v,v\rangle \leq 1,\nonumber\\
		& && R_{\ind{I}}=\langle u,A_1(I_1)\dots A_t(I_t)v\rangle,\ \mathrm{for\ }\ind{I}\in([n]\times [n])^t.\nonumber
	\end{align}	
        We finally claim that the above optimization problem is equivalent to 
        \begin{align}
		&\sup && 2|\sum_{\ind{I}\in ([n]\times [n])^t}T_{\ind{I}}R_{\ind{I}}|,\label{eq:SDPb}\\ \nonumber
		&\mathrm{s.t.} &&  R\in (\mathbb C^{n\times n})^{t},\  d\in\mathbb{N},\ u,v\in \mathbb{C}^d, A_s\in M_d\ \mathrm{contractions\ for}\ s\in [n],\\ \nonumber
		& &&\langle u,u\rangle,\langle v,v\rangle \leq 1/2,\\ \nonumber
		& && R_{\ind{I}}=\langle u,A_1(I_1)\dots A_t(I_t)v\rangle,\ \mathrm{for\ }\ind{I}\in([n]\times [n])^t.
	\end{align}	
	We first note that the value of \cref{eq:SDPa} is greater or equal than the one of \cref{eq:SDPb}, because the feasible region is larger in the case of \cref{eq:SDPa}. On the other hand, if one picks a feasible instance $(u,v,A)$ of \cref{eq:SDPa}, one can define the instance $(\tilde u,\tilde v,A)$ by $$\tilde u=\frac{u\sqrt{\norm{u}^2+\norm{v}^2}}{2\norm{u}},\ \tilde v=\frac{v\sqrt{\norm{u}^2+\norm{v}^2}}{2\norm{v}},$$
 which is feasible for \cref{eq:SDPb} and attains a value greater or equal than $(u,v,A)$, because 
 \begin{align*}
     |\sum T_{\ind I}\langle \tilde u,A_1(I_1)\dots A_t(I_t)\tilde v\rangle|&=\frac{\sqrt{\norm{u}^2+\norm{v}^2}}{2\norm{u}\norm{v}}|\sum T_{\ind I} \langle u,A_1(I_1)\dots A_t(I_t) v\rangle|\\
     &\geq |\sum T_{\ind I}\langle u,A_1(I_1)\dots A_t(I_t) v\rangle|.
 \end{align*}
    Now, the result follows from the fact that the optimal value of  \cref{eq:SDPb} is $\norm{T}_{\cb}$
\end{proof}

\subsection{Strong duality}
Finally, we prove \cref{theo:main} \eqref{item:main3}. Before that, we formally state what is the condition that ensures strong duality.
\begin{theorem}[Slater's theorem]\label{theo:slater}
    Let $(P)$ and $(D)$ be a primal-dual pair of semidefinite programs, as in \cref{eq:SDPDef}. Assume that $(P)$ is feasible and there exists a strictly positive instance for $(D)$, i.e., there exists  $y\in \mathbb{R}^L$ such that $C-\mathcal{B}^*(y)$ is strictly positive. Then the optimal values of $(P)$ and $(D)$ are equal.
\end{theorem}

\begin{proof}[ of \cref{theo:CS}  \eqref{item:main3}]
	By adding slack variables one can write \cref{eq:TcbinSDP,eq:CSinSDP} in their canonical form as in \cref{eq:SDPDef}. Then, it is (tedious and) routine to check that $(D_{CS})$ is the dual of $(P_{CS})$. Also, the conditions of \cref{theo:slater} are satisfied, because $(P_{CS})$ is feasible, as every $T$ factors as in \cref{eq:CSfactorization} for some $u,v$ with sufficiently big norm, and the following parameters define a strictly positive feasible instance\footnote{Which in the non-canonical form that $(D_{CS})$ is expressed means that all inequalities are strict.} for $(D_{CS})$ 
	\begin{align*}
		y_0=y_0'&=\frac{1}{3},\\
		y_{\ind I,\ind J}&=\frac{\delta_{\ind I,\ind J}}{3(n+1)^s},\ \mathrm{for\ }\ind I,\ind J\in ([n]\otimes [n])^{s},\ s\in [t],\\
            y_{0,\ind I}=y_{\ind I,0}&=0, \mathrm{for\ }\ind I\in ([n]\otimes [n])^{s},\ s\in [t].
	\end{align*}
\end{proof}

\noindent \textbf{Acknowledgements.} We thank Jop Bri\"et, Sander Gribling, Harold Nieuwboer, Carlos Palazuelos and Vern Paulsen for useful comments and discussions. 
\bibliographystyle{alphaurl}
\bibliography{Bibliography}

\begin{thebibliography}{NRV13}

\bibitem[ABP19]{QQA=CBF}
Srinivasan Arunachalam, Jop Bri{\"{e}}t, and Carlos Palazuelos.
\newblock Quantum query algorithms are completely bounded forms.
\newblock {\em SIAM J.\ Comput}, 48(3):903--925, 2019.
\newblock Preliminary version in ITCS'18.

\bibitem[BP91]{BLECHER1991262}
David~P Blecher and Vern~I Paulsen.
\newblock Tensor products of operator spaces.
\newblock {\em Journal of Functional Analysis}, 99(2):262--292, 1991.
\newblock URL:
  \url{https://www.sciencedirect.com/science/article/pii/0022123691900424},
  \href {https://doi.org/10.1016/0022-1236(91)90042-4}
  {\path{doi:10.1016/0022-1236(91)90042-4}}.

\bibitem[CS87]{christensen1987representations}
Erik Christensen and Allan~M Sinclair.
\newblock Representations of completely bounded multilinear operators.
\newblock {\em Journal of Functional analysis}, 72(1):151--181, 1987.

\bibitem[GL19]{gribling2019semidefinite}
Sander Gribling and Monique Laurent.
\newblock Semidefinite programming formulations for the completely bounded norm
  of a tensor.
\newblock {\em arXiv preprint arXiv:1901.04921}, 2019.

\bibitem[LR05]{LAURENT2005393}
Monique Laurent and Franz Rendl.
\newblock Semidefinite programming and integer programming.
\newblock In K.~Aardal, G.L. Nemhauser, and R.~Weismantel, editors, {\em
  Discrete Optimization}, volume~12 of {\em Handbooks in Operations Research
  and Management Science}, pages 393--514. Elsevier, 2005.
\newblock URL:
  \url{https://www.sciencedirect.com/science/article/pii/S0927050705120088},
  \href {https://doi.org/10.1016/S0927-0507(05)12008-8}
  {\path{doi:10.1016/S0927-0507(05)12008-8}}.

\bibitem[NRV13]{naor2013efficient}
Assaf Naor, Oded Regev, and Thomas Vidick.
\newblock Efficient rounding for the noncommutative grothendieck inequality.
\newblock In {\em Proceedings of the forty-fifth annual ACM symposium on Theory
  of computing}, pages 71--80, 2013.

\bibitem[PS87]{PAULSEN1987258}
V.I Paulsen and R.R Smith.
\newblock Multilinear maps and tensor norms on operator systems.
\newblock {\em Journal of Functional Analysis}, 73(2):258--276, 1987.
\newblock URL:
  \url{https://www.sciencedirect.com/science/article/pii/0022123687900681},
  \href {https://doi.org/10.1016/0022-1236(87)90068-1}
  {\path{doi:10.1016/0022-1236(87)90068-1}}.

\bibitem[RV14]{regev2014elementary}
Oded Regev and Thomas Vidick.
\newblock Elementary proofs of grothendieck theorems for completely bounded
  norms.
\newblock {\em Journal of operator theory}, pages 491--505, 2014.

\bibitem[RV15]{regev2015quantum}
Oded Regev and Thomas Vidick.
\newblock Quantum xor games.
\newblock {\em ACM Transactions on Computation Theory (ToCT)}, 7(4):1--43,
  2015.

\bibitem[Wat09]{watrous2009semidefinite}
John Watrous.
\newblock Semidefinite programs for completely bounded norms.
\newblock {\em arXiv preprint arXiv:0901.4709}, 2009.

\bibitem[Wat12]{watrous2012simpler}
John Watrous.
\newblock Simpler semidefinite programs for completely bounded norms.
\newblock {\em arXiv preprint arXiv:1207.5726}, 2012.

\bibitem[Wat21]{watrouslecturenotes}
John Watrous.
\newblock Conic programming (lecture notes).
\newblock 2021.
\newblock URL:
  \url{https://johnwatrous.com/wp-content/uploads/2023/08/QIT-notes.01.pdf}.

\end{thebibliography}
\end{document}